\documentclass{amsart}
\usepackage[utf8]{inputenc}

\title[A multi-point maximum principle for Harnack inequalities]{A multi-point maximum principle to prove global Harnack inequalities for Schrödinger operators}
\author{Ben Andrews, Daniel Hauer and Jessica Slegers}
\date{\today}

\address{Ben Andrews, Mathematical Sciences Institute, Australian National University, Canberra, ACT, 2006, Australia }
\email{Ben.Andrews@anu.edu.au}
\address{Daniel Hauer, Brandenburg University of Technology Cottbus–Senftenberg, Platz der Deutschen Einheit 1, 03046 Cottbus, Germany; School of Mathematics and Statistics, The University of Sydney, Sydney, NSW, 2006, Australia }
\email{daniel.hauer@b-tu.de}
\address{
    Jessica Slegers, School of Mathematics and Statistics, The University of Sydney, Sydney, NSW, 2006, Australia; Brandenburg University of Technology Cottbus–\linebreak Senftenberg, Platz der Deutschen Einheit 1, 03046 Cottbus, Germany}
\email{jessica.slegers@sydney.edu.au}

\usepackage{amsmath,amssymb}
\usepackage{bbm}
\usepackage{mathtools}
\usepackage[colorlinks,citecolor=blue,urlcolor=blue]{hyperref}

\usepackage{amsthm}

\newtheoremstyle{theorem}{}{}{\itshape}{}{\bfseries}{.}{ }{}
\theoremstyle{theorem}
\newtheorem{theorem}{Theorem}[section]

\newtheorem{lemma}[theorem]{Lemma}
\newtheorem{corollary}[theorem]{Corollary}

\newtheoremstyle{definition}{}{}{}{}{\bfseries}{.}{ }{}
\theoremstyle{definition}
\newtheorem{definition}{Definition}[section]
\newtheorem{remark}{Remark}[section]
\newtheorem{example}{Example}[section]

\numberwithin{equation}{section} 

\newcommand{\R}{\mathbb{R}}

\newcommand{\C}{\mathcal{C}}
\newcommand{\dd}{\ \mathrm{d}}

\DeclareMathOperator*{\Ric}{Ric}

\thanks{The first author's research was supported by the Laureate Grant FL150100126 and Discovery Projects Grant DP250103808 of the Australian Research Council. The second author's research was supported by the Australian Research Council grant DP220100067. The third author's research was supported by an Australian Government Research Training Program (RTP) Scholarship. The authors are grateful for the support.}

\keywords{Harnack inequality, heat equation, Li-Yau estimate, maximum principle, multi-point maximum principle, Schrödinger operator, Ornstein-Uhlenbeck operator with potential}
 
\begin{document}

\begin{abstract}
In this article, we introduce a new methodology to prove global parabolic Harnack inequalities on Riemannian manifolds. We focus on presenting a new proof of the global pointwise Harnack inequality satisfied by positive solutions of the linear Schrödinger equation on a Riemannian manifold $M$ with nonnegative Ricci curvature, where the potential term $V$ is bounded from below. Our approach is based on a multi-point maximum principle argument. Standard proofs of this result (see, for instance, Li-Yau [Acta Math, 1986]) rely on first establishing a gradient estimate. This requires the solution to be at least $C^4$ on $M$. We instead prove the Harnack inequality directly, which has the advantage of avoiding higher-order derivatives of the solution in the proof, enabling us to assume it is only $C^2$ on $M$. In the particular case that $V$ is the quadratic potential $V(x)=|x|^2$ and $M$ is the Euclidean space $\mathbb{R}^d$, we prove a new Harnack inequality with sharper constants. Finally, we treat positive solutions of the Schrödinger equation with a gradient drift term, including applications to the Ornstein-Uhlenbeck operator $\Delta - x\cdot \nabla$ with quadratic potential in $\mathbb{R}^d$.
\end{abstract}

\maketitle

\section{Introduction and main results}
An important milestone in the study of the heat equation 
\begin{equation}\label{eq:schroedinger}
        \partial_t u = \Delta u - V u \qquad \text{in $M\times (0,\infty)$}
    \end{equation}
governed by the Schrödinger operator $-\Delta + V$ on a Riemannian manifold $M$ was the result due to Li and Yau \cite{LiYau1986}, who demonstrated that every positive solution $u$ of \eqref{eq:schroedinger} on a complete Riemannian manifold $M$ with nonnegative Ricci curvature satisfies the global Harnack inequality
\begin{equation}\label{eq:li-yau-harnack}
      u(x,t) \geq u(y,s) \left(\frac{s}{t} \right)^{d/2} e^{-C_2(\theta^{2/3}+C_1^{1/2})(t-s)-\omega(x,y;t,s)}  
    \end{equation}
for all $x,y \in M$ and $0<s<t$. Here, $V=V(x,t) \in C^{2,1}(M \times (0,\infty))$, the constant $C_1$ is such that $\Delta V \leq C_1$ on $M \times (0,\infty)$, $C_2$ is a constant depending on the dimension $d$ of the manifold $M$, $\theta$ describes the growth of $\nabla V$, and 
\[
\omega(x,y;t,s):= \inf_{\gamma \in \Gamma_{x,y}} \frac{1}{4(t-s)}\int_0 ^1 |\dot{\gamma}|^2 \dd \tau + (t-s) \int_0^1 V(\gamma(\tau),(1-\tau)s + \tau t) \dd \tau,
\]
where $\Gamma_{x,y}$ denotes the set of curves $\gamma:[0,1] \to M$ with $\gamma(0)=y$ and $\gamma(1)=x$. The Harnack inequality \eqref{eq:li-yau-harnack} has several important consequences, for example, it implies sharp bounds on the heat kernel (see \cite{LiYau1986}), Hölder continuity of solutions of \eqref{eq:schroedinger}, and eigenvalue estimates (see \cite{Davies-Heat-Kernels}). 

In this paper, we are mainly interested in providing a new methodology of the proof of this result. The main idea of Li and Yau is to take $v=\log u$ and apply the Laplace operator to the differential equation 
\[
\partial_t v = \Delta v + |\nabla v|^2
\]
solved by $v$. By applying the maximum principle, a local gradient estimate is first obtained on a ball of radius $R$. Taking $R \to \infty$ then produces the gradient estimate 
\begin{equation}\label{eq:li-yau-diff-harnack}
    -\Delta (\log u) = \frac{|\nabla u|^2}{u^2} - \frac{\partial_t u}{u}-V \leq \frac{d}{2t}  + C_2 \theta^{2/3} + \left(\frac{d}{2} C_1\right)^{1/2}
\end{equation}
holding globally on $M \times (0,\infty)$. The inequality \eqref{eq:li-yau-diff-harnack}, often referred to as the \emph{Li-Yau estimate}, describes as the \emph{differential form of the Harnack inequality}, meaning that performing space-time integration on \eqref{eq:li-yau-diff-harnack} yields the Harnack inequality \eqref{eq:li-yau-harnack}. We note that this procedure involves the fourth-order derivatives of the solution $u$.

The results obtained in \cite{LiYau1986} have been improved by various authors using similar techniques, especially for the heat equation ($V\equiv 0$) (see, for instance, \cite{Bakry-2017, Bakry-Qian-1999, Davies-Heat-Kernels, Hamilton-1993, Li-Xu}), but also in the general potential case \cite{li-1991, Ruan, Yau-1994}. The approach by Li and Yau has since been used extensively to prove Harnack inequalities in several other contexts, including but not limited to operators of porous medium \cite{Aronson-Benilan-1979, Crandall-Pierre, Huang-Huang-Li, Lu-Ni-Vazquez-Villani, Ma-Li, Yau-1994}, $p$-Laplace \cite{Auchmuty-Bao-1994, Esteban-Vasquez-1990, Kotschwar-Ni} or doubly nonlinear type \cite{Esteban-Vazquez-1988, Wang-Chen-doubly-nonlinear} and involving lower-order terms \cite{Ma-Zhao-Song, Wang-2017, Yau-1995}, as well as geometric flows \cite{Chow-Harnack-Gauss, Chow-Harnack-Yamabe, Hamilton-Ricci-Harnack, Hamilton-Mean-Harnack}. 

We acknowledge that besides the argument by Li and Yau \cite{LiYau1986}, other methods have been established in the literature to prove parabolic Harnack inequalities. One main standard approach is Moser iteration, which was originally applied to derive Harnack inequalities for positive solutions of parabolic problems involving divergence-form operators on the Euclidean space $\R^d$ (see \cite{Moser1964}). This technique has since been adapted under suitable conditions to the settings of manifolds \cite{Grigoryan-1991, Saloff-Coste-1992} and metric measure spaces \cite{Kinnunen-metric-measure}, and more recently to nonlocal equations (see, for instance, \cite{Ding-2021, Garain-Kinnunen-2023, Kassmann-Weidner-2024-2, Kassmann-Weidner-2024, Strömqvist}). While Moser iteration remains an important tool to prove Harnack inequalities in various contexts, one limitation is that it typically does not produce sharp estimates. Other techniques available to prove parabolic Harnack inequalities include, but are not limited to, the original heat-potential based approach of Hadamard \cite{Hadamard1954} and Pini \cite{Pini1954}, the probabilistic approach employed by Krylov and Safonov \cite{Krylov-Safonov-1981} for non-divergence form operators, and the method of intrinsic scaling developed by DiBenedetto \cite{DiBenedetto-Harnack-Book} to treat degenerate and singular equations. For a more thorough introduction to the topic, one may refer to the survey paper of Kassmann \cite{Kassmann-Harnack-survey}.

In this article, we aim to introduce a novel method to prove parabolic Harnack inequalities, which sharpens the constants in \eqref{eq:li-yau-harnack}. Our approach is based on a multi-point maximum principle argument and allows us to directly prove Harnack inequalities. The differential form can be recovered afterwards using space-time differentiation (see Section \ref{sec:diff-harnack}). As a consequence, we only require the solution $u$ to be $C^2$ on $M$. This is advantageous, because in order to proceed via the approach by Li and Yau \cite{LiYau1986}, one requires $u$ to be at least $C^4$ on $M$. While this is not an obstacle in the linear case, since solutions of \eqref{eq:schroedinger} are generally smooth, this can lead to complications when studying solutions of nonlinear equations.

We first consider positive solutions of the Neumann problem associated to the Schrödinger equation \eqref{eq:schroedinger} on compact manifolds, possibly with boundary, and later obtain a result on complete manifolds $M$ by approximating $M$ by an increasing sequence of compact domains. While we restrict the scope of this article to linear parabolic equations, we note that our methods can also be applied to derive Harnack inequalities associated to doubly nonlinear equations, such as the porous medium equation and heat equation corresponding to the $p$-Laplace operator. The details will appear in the doctoral thesis of the third author. In addition, the methodology can also be used to derive local Harnack inequalities for the aforementioned equations.

\bigskip

Our main theorems concern sufficient conditions for a global Harnack inequality to be satisfied by positive solutions $u$ of 
\begin{equation}\tag{\ref{eq:schroedinger}}
    \begin{alignedat}{2}
       \partial_t u &= \Delta u -Vu && \qquad \text{in $M \times (0,\infty)$},
    \end{alignedat}
\end{equation}
where the manifold $M$ is either assumed to be complete or compact with or without boundary. In particular, if $V$ is bounded from below, the problem \eqref{eq:schroedinger} admits a positive solution $u$ for every positive initial datum $u_0$, for instance, $u_0 \in L^2(M)$ (see, for example, \cite{kato}).

A key role is played in the proof by the energy functional $E$ defined by 
\begin{equation}\label{eq:definition-of-E}
E[\gamma;t,s]:= \frac{1}{4(t-s)}\int_0 ^1 |\dot{\gamma}|^2 \dd \tau + (t-s) \int_0^1 V(\gamma(\tau)) \dd \tau   
\end{equation} 
for every $C^1$ curve $\gamma:[0,1] \to M$. We are particularly interested in the minimum value of $E[\gamma;t,s]$ over the set \[\Gamma_{x,y}:= \{ \gamma \in C^1([0,1],M) \mid \gamma(0)=y, \gamma(1)=x \}\] for every $x,y \in M$ and $(t,s)$ in the set \[S:= \{ (t,s) \in \R ^2 \mid 0 < s < t\}.\] We later denote this minimum value by $\omega(x,y;t,s)$. We note that the energy defined in \eqref{eq:definition-of-E} has already been studied extensively, especially by physicists (see, for example, \cite{Carmona-Simon}), who refer to the function $\omega$ as the \emph{Agmon metric}. In particular, it is well-known that energies of the type \eqref{eq:definition-of-E} do indeed attain a minimum over $\Gamma_{x,y}$ (see \cite{Lemaire-1977}).

\begin{definition}
A curve $\gamma_0 \in \Gamma_{x,y}$ is called a \emph{$V$-geodesic} if there exist $0<s<t$ such that \[
E[\gamma_0;t,s] \leq E[\gamma;t,s].
\] for all $\gamma \in \Gamma_{x,y}$. In addition, we say that a compact manifold $M$ is \emph{$V$-convex} or has a \emph{$V$-convex boundary} if 
\begin{equation}\label{eq:V-convexity-manifold}
    h(\xi,\xi) + g(\nabla V(x), \nu(x)) > 0
\end{equation}
holds for all $x \in \partial M$ and $\xi \in T_x \partial M \setminus \{0\}$, where $h$ denotes the second fundamental form on $\partial M$ and $\nu(x)$ is the outward pointing unit normal vector to the boundary at $x \in \partial M$. 
\end{definition}

\begin{remark}\label{rem:V-convex}
\begin{enumerate}
\item[]
\item The condition \eqref{eq:V-convexity-manifold} implies the nonnegativity of the second fundamental form and therefore every $V$-convex domain is necessarily convex. In addition, \eqref{eq:V-convexity-manifold} implies that $\nabla V \cdot \nu \geq 0$ on $\partial M$, which is a condition previously assumed by Yau \cite{Yau-1995}.
\item By applying a maximum principle argument to analyse the distance of a $V$-geodesic from the boundary $\partial M$, one can show that \eqref{eq:V-convexity-manifold} is sufficient to ensure that all $V$-geodesics connecting two points $x,y \in M$ are contained in $M$, and that the endpoints of a $V$-geodesic cannot intersect the boundary tangentially.
\end{enumerate} 
\end{remark}

\begin{definition}
A manifold $M$ is called \emph{$V$-approximable} if there exists an increasing sequence $(M_n)_{n \geq 1}$ of $V$-convex regions $M_n \subset M$ such that $\bigcup_{n \geq 1} M_n = M$. 
\end{definition}

For example, if $V$ is geodesically convex and has a minimiser at $x \in M$, then $M$ is $V$-approximable via geodesic balls centred at $x$ and of increasing radius.

\bigskip

Our main result is as follows.

\begin{theorem}\label{thm:Manifold-Schroedinger-full}
Let $(M,g)$ be a complete Riemannian manifold without boundary and nonnegative Ricci curvature. Let $V \in C^2(M)$ be bounded from below and such that $M$ is $V$-approximable. Suppose there exist strictly increasing, differentiable functions $\beta:[0,\infty) \to [0,\infty)$ satisfying $\beta(0)=0$ and $A:[0,\infty) \to [0,\infty)$ such that 
\begin{equation}\label{eq:second-order-assumption-manifold}
\begin{split}
   (t-s)\int_0^1 \frac{d}{2}A'(\tau t + (1-\tau)s)^2 &+ A(\tau t + (1-\tau)s)^2 \Delta V(\gamma_0) \dd \tau \\&\leq A(t)^2 \frac{\beta'(t)}{\beta(t)} - A(s)^2 \frac{\beta'(s)}{\beta(s)}  
\end{split}
\end{equation}
holds for all $V$-geodesics $\gamma_0$ and $0<s<t$. 
Then every positive solution $u$ of \eqref{eq:schroedinger} satisfies
\begin{equation}\label{eq:heat-harnack-manifold}
u(x,t) \geq u(y,s) \frac{\beta(s)}{\beta(t)} e^{-\omega(x,y;t,s)}.
\end{equation}
for all $x,y \in M$ and $0 < s < t$, where $\omega:M \times M \times S$ is given by
\begin{equation}\label{eq:omega-def}
    \omega(x,y;t,s):= \min_{\gamma \in \Gamma_{x,y}} \frac{1}{4(t-s)}\int_0 ^1 |\dot{\gamma}|^2 \dd \tau + (t-s) \int_0^1 V(\gamma(\tau)) \dd \tau.
\end{equation}
\end{theorem}
We note that the function $\omega$ defined in \eqref{eq:omega-def} coincides with the function studied by Li and Yau in \cite{LiYau1986} and we revisit it further below.

In order to prove this theorem, we approximate the positive solution $u$ of \eqref{eq:schroedinger} by a sequence of positive solutions $u_n$ of the Neumann problem \eqref{eq:schroe-neumann-manifold} defined below on an increasing sequence of compact $V$-convex submanifolds. In particular, we require the following result regarding positive solutions $u$ of the Neumann problem.

\begin{theorem}\label{thm:Manifold-Schroedinger}
Let $(M,g)$ be a compact Riemannian manifold with nonnegative Ricci curvature and (possibly empty) boundary $\partial M$. Let $V \in C^2(M)$ be bounded from below and such that $\partial M$ is $V$-convex. In addition, suppose that \eqref{eq:second-order-assumption-manifold} holds for all $V$-geodesics $\gamma_0$. Then every positive solution $u$ of  
\begin{equation}\label{eq:schroe-neumann-manifold}
   \begin{cases}
   \begin{alignedat}{2}
       \partial_t u &= \Delta u -V u&& \qquad \text{in $M \times (0,\infty)$,} \\
        \nabla u \cdot \nu &= 0 && \qquad \text{on $\partial M \times (0,\infty)$}
   \end{alignedat}
\end{cases} 
\end{equation}
satisfies the Harnack inequality \eqref{eq:heat-harnack-manifold} for all $x,y \in M$ and $0<s<t$.
\end{theorem}

\begin{remark}
Without loss of generality, we may always assume that the potential $V$ in  \eqref{eq:schroedinger} is nonnegative. To see this, note that if $V$ is merely bounded from below by $-\alpha$ for some $\alpha>0$, applying the transformation $u \mapsto ue^{-\alpha t}$ to the equation \eqref{eq:schroedinger} yields 
    \[
    \partial_t u = \Delta u - (V_\alpha) u
    \]
    with $V_\alpha:= V + \alpha \geq 0$. After obtaining a result for this new, nonnegative potential $V+\alpha$, we can reverse the transformation to give a Harnack inequality for solutions of the original equation.
\end{remark}

After making a change of variables, Theorem \ref{thm:Manifold-Schroedinger-full} immediately implies the following result for positive solutions of the equation \eqref{eq:schroedinger} perturbed by a gradient drift term. 

\begin{corollary}\label{cor:drift-full}
Let $(M,g)$ be a complete Riemannian manifold without boundary and nonnegative Ricci curvature. Let $f \in C^4(M)$, \mbox{$V \in C^2(M)$}, and let $u$ be a positive solution of 
\begin{equation}\label{eq:heat-grad-drift-potential-full}
\partial_t u = \Delta u -2\nabla f \cdot \nabla u - Vu \qquad \text{in $M\times (0,\infty)$}.
\end{equation}
If the conditions of Theorem \ref{thm:Manifold-Schroedinger-full} hold for the potential $\tilde{V}:= |\nabla f|^2 - \Delta f + V$, then $u$ satisfies
    \begin{equation}\label{eq:drift-transformed-harnack-og}
u(x,t) \geq u(y,s) \frac{\beta(s)}{\beta(t)} e^{f(x)-f(y)-\tilde{\omega}(x,y;t,s)}
\end{equation}
for all $x,y \in M$ and $0<s<t$, where 
\[
\tilde{\omega}(x,y;t,s):= \min_{\Gamma_{x,y}} \frac{1}{4(t-s)}\int_0 ^1 |\dot{\gamma}|^2 \dd \tau + (t-s) \int_0^1 \tilde{V}(\gamma(\tau)) \dd \tau.
\]
\end{corollary}

By applying Theorem \ref{thm:Manifold-Schroedinger} instead of Theorem \ref{thm:Manifold-Schroedinger-full}, we obtain a similar result on compact manifolds.

\begin{corollary}\label{cor:drift}
Let $(M,g)$ be a compact Riemannian manifold with nonnegative Ricci curvature and (possibly empty) boundary $\partial M$. Let $f \in C^4(M)$, \mbox{$V \in C^2(M)$}, and let $u$ be a positive solution of 
\begin{equation}\label{eq:heat-grad-drift-potential}
   \begin{cases}
   \begin{alignedat}{2}
   \partial_t u &= \Delta u -2\nabla f \cdot \nabla u - Vu &&\qquad \text{in $M \times (0,\infty)$,} \\
   \nabla u \cdot \nu &= (\nabla f \cdot \nu)u &&\qquad \text{on $\partial M \times (0,\infty)$.}
   \end{alignedat}
\end{cases} 
\end{equation}
If the conditions of Theorem \ref{thm:Manifold-Schroedinger} hold for the potential $\tilde{V}:= |\nabla f|^2 - \Delta f + V$, then $u$ satisfies \eqref{eq:drift-transformed-harnack-og}
for all $x,y \in M$ and $0<s<t$.
\end{corollary}

In the Euclidean case $M=\R^d$, we encounter fewer technical restrictions in the proof, which allows us to generalise the function $\omega$ defined in \eqref{eq:omega-def}. The following theorem is our main result in this case.

\begin{theorem}\label{thm:Euclidean-Schroedinger-full}
Let $V \in C^2(\R^d)$ be bounded from below. Suppose there exists a continuous function $\omega=\omega(x,y;t,s): \R^d \times \R^d \times S \rightarrow \R$, which is twice differentiable in the first and second arguments and differentiable in the third and fourth arguments. Suppose that $\omega(x,y;t,s) \rightarrow \infty$ whenever $t \rightarrow s^+$ for $x\ne y$ and $\lim_{t\to s^+} \omega(x,x;t,s) \geq 0$ for all $x \in \R^d$. In addition, assume that $\omega$ satisfies
\begin{align}
    \partial_t \omega + |\nabla_x\omega|^2 &\geq V(x) \ \ \qquad \text{ in $\R^d \times \R^d \times S$}, \label{eq:schroe-omega-ineq-1}\\
    \partial_s \omega - |\nabla_y\omega|^2 &\geq -V(y) \qquad \text{ in $\R^d \times \R^d \times S$,}\label{eq:schroe-omega-ineq-2}
\end{align}
and that there exist strictly increasing, differentiable functions $\beta:[0,\infty) \to [0,\infty)$ satisfying $\beta(0)=0$ and $A:[0,\infty) \to [0,\infty)$ such that
\begin{equation}
    A(t)^2\Delta_x\omega +A(s)^2\Delta_y \omega + 2A(t)A(s) \sum_{i=1}^{d} \omega_{x_i y_i} \leq A(t)^2 \frac{\beta'(t)}{\beta(t)}-A(s)^2 \frac{\beta'(s)}{\beta(s)}\label{eq:schroe-omega-ineq-3}
\end{equation}
holds in $\R^d \times \R^d \times S$.
Finally, assume there exists an increasing sequence $(\Omega_{n})_{n \geq 1}$ of smooth, bounded domains $\Omega_{n} \subset \R^d$ such that $\bigcup_{n \geq 1} \Omega_{n} = \R^d$ and 
\begin{equation}\label{eq:V-convexity}
    \begin{split}
    \nabla_x \omega \cdot \nu(x) \geq 0, \qquad \text{for all $x \in \partial \Omega_n$, $y \in \overline{\Omega_n}$, and $0<s<t$,}\\
    \nabla_y \omega \cdot \nu(y) \geq 0, \qquad \text{for all $x \in \overline{\Omega_n}$, $y \in \partial \Omega_n$ and $0<s<t$.}
\end{split}
\end{equation}
hold for all $ n \geq 1$. Then, every positive solution $u$ of 
\begin{equation}\label{eq:schroedinger-full-euclidean}
   \begin{alignedat}{2}
       \partial_t u &= \Delta u -V u&& \qquad \text{in $\R^d \times (0,\infty)$}
   \end{alignedat} 
\end{equation} satisfies
\begin{equation}\label{eq:heat-harnack-omega}
u(x,t) \geq u(y,s) \frac{\beta(s)}{\beta(t)} e^{-\omega(x,y;t,s)}
\end{equation}
for all $x,y \in \R^d$ and $0 <s<t$.
\end{theorem}

We emphasise that these inequalities \eqref{eq:schroe-omega-ineq-1}--\eqref{eq:V-convexity} are consistent with the assumptions of Theorem \ref{thm:Manifold-Schroedinger-full} if one takes $\omega$ to be the function defined in \eqref{eq:omega-def}. In particular, the function $\omega$ in \eqref{eq:omega-def} satisfies \eqref{eq:schroe-omega-ineq-1} and \eqref{eq:schroe-omega-ineq-2} with equality (see \cite[Section 3]{LiYau1986}), inequality \eqref{eq:schroe-omega-ineq-3} is implied by \eqref{eq:second-order-assumption-manifold}, and the existence of domains $\Omega_n$ on which the inequalities \eqref{eq:V-convexity} hold replaces the $V$-approximability assumption in Theorem \ref{thm:Manifold-Schroedinger-full}. Indeed, if each of the domains $\Omega_n \subset \R^d$ is $V$-convex and $\gamma_0:[0,1] \to \overline{\Omega_n}$ is a $V$-geodesic with $\gamma_0(0)=y$ and $\gamma_0(1)=x$, then the inequalities 
\[\nabla_x \omega \cdot \nu = \frac{\dot{\gamma}_0(1) \cdot \nu}{2(t-s)} \geq 0 \quad \text{and} \quad \nabla_y \omega \cdot \nu = -\frac{\dot{\gamma}_0(0) \cdot \nu}{2(t-s)} \geq 0\]
both hold, that is, \eqref{eq:V-convexity} will be satisfied.

As in the manifold case, Theorem \ref{thm:Euclidean-Schroedinger-full} follows by an approximation argument after obtaining the following result for the Neumann problem.

\begin{theorem}\label{thm:Euclidean-schrödinger}
Let $\Omega \subset \R^d$ be a bounded domain with smooth boundary $\partial \Omega$ and let $V \in C^2(\Omega)$ be bounded from below.
Let $\omega=\omega(x,y;t,s): \overline{\Omega} \times \overline{\Omega} \times S \rightarrow \R$ be a continuous function, which is twice differentiable in the first and second arguments and differentiable in the third and fourth arguments. Suppose that $\omega(x,y;t,s) \rightarrow \infty$ whenever $t \rightarrow s^+$ for $x\ne y$ and $\lim_{t\to s^+} \omega(x,x;t,s) \geq 0$ for all $x \in \overline{\Omega}$. In addition, assume that $\omega$ satisfies \eqref{eq:schroe-omega-ineq-1}--\eqref{eq:schroe-omega-ineq-3} in $\overline{\Omega} \times \overline{\Omega} \times S$. Finally, assume that 
\begin{equation}\label{eq:V-convexity-compact}
    \begin{split}
    \nabla_x \omega \cdot \nu(x) \geq 0, \qquad \text{for all $x \in \partial \Omega$, $y \in \overline{\Omega}$, and $0<s<t$,}\\
    \nabla_y \omega \cdot \nu(y) \geq 0, \qquad \text{for all $x \in \overline{\Omega}$, $y \in \partial \Omega$ and $0<s<t$.}
\end{split}
\end{equation}
Then every positive solution $u$ of 
\begin{equation}\label{eq:schroe-neumann-rd}
   \begin{cases}
   \begin{alignedat}{2}
       \partial_t u &= \Delta u -V u&& \qquad \text{in $\Omega \times (0,\infty)$,} \\
        \nabla u \cdot \nu &= 0 && \qquad \text{on $\partial \Omega \times (0,\infty)$}
   \end{alignedat}
\end{cases} 
\end{equation} satisfies \eqref{eq:heat-harnack-omega} for all $x,y \in \overline{\Omega}$ and $0<s<t$.
\end{theorem}

We begin in Section \ref{sec:euclidean} by studying the Euclidean case, followed by the manifold case in Section \ref{sec:schroedinger}. In Section \ref{sec:grad-drift}, we briefly consider adding a gradient drift term to the equation \eqref{eq:schroedinger}. Here, our findings follow quickly from the results of Section \ref{sec:schroedinger}. Finally, in Section \ref{sec:diff-harnack}, we demonstrate how to recover a differential Harnack inequality from the Harnack inequality \eqref{eq:heat-harnack-manifold}.


\section{The Euclidean case}\label{sec:euclidean}

\subsection{Proofs of Theorems \ref{thm:Euclidean-Schroedinger-full} and \ref{thm:Euclidean-schrödinger}}
In order to introduce the idea of our method in the simplest possible setting, we begin by considering the case where $M=\Omega \subset \R^d$ is a bounded domain with smooth boundary.

\begin{proof}[Proof of Theorem \ref{thm:Euclidean-schrödinger}]
    Let $u$ be a positive solution of \eqref{eq:schroe-neumann-rd} and set $v=\log u$. Then $v$ is a solution to
    \begin{equation*}\label{eq:schroe-neumann-v}
   \begin{cases}
        \partial_t v = \Delta v + |\nabla v|^2 -V\quad &\text{in $\Omega \times (0,\infty)$,} \\
        \nabla v \cdot \nu = 0 \quad &\text{on $\partial \Omega \times (0,\infty)$}.
    \end{cases} 
\end{equation*}
Fix $\varepsilon > 0$ and define a function $Z$ on $\overline{\Omega} \times \overline{\Omega} \times S$ by
    \[Z(x,y;t,s)=v(x,t) - v(y,s) + \log\left(\frac{\beta(t)}{\beta(s)}\right) + \omega(x,y;t,s) + \frac{\varepsilon}{2}(t-s)^2 + \varepsilon.\]

We make some initial observations about $Z$ for $(t,s)$ near the boundary of $S$. Firstly, as $s \to 0^+$, the term $\log\left(\frac{\beta(t)}{\beta(s)}\right) \to \infty$ due to the assumption that $\beta(0)=0$. Therefore, $Z>0$ for $s$ near 0. Next, consider the behaviour of $Z$ as $t\to s^+$. If $x \ne y$, then the assumption that $\omega \to \infty$ as $t \to s^+$ guarantees that $Z>0$ for $t$ near $s$. Otherwise, if $x=y$, then the difference $v(x,t)-v(x,s) \to 0$, so we can assume $v(x,t)-v(x,s) + \varepsilon > 0$. By assumption, all other terms in the limit of $Z$ are nonnegative, so once again, $Z>0$ for $t$ near $s$. 

Now, let $\rho=t^2 + s^2$. As $\rho \rightarrow 0^+$, that is, when $(t,s)$ is in a small neighbourhood of $(0,0)$ in $S$, $Z>0$ by the same argument as for when $t \to s^+$. Then, for $\rho >0$, $Z$ either remains strictly positive, or there exists $\rho_0:= t_0^2 + s_0^2>0$ and a point $P_0:=(x_0,y_0;t_0,s_0) \in \overline{\Omega} \times \overline{\Omega} \times S$ such that $Z(P_0)=0$. We show that the latter cannot occur.

Suppose, by contradiction, that such a point $P_0$ exists. We assume that $t_0$ and $s_0$ are taken so that $\rho_0$ is minimal, that is, $Z(x,y;t,s)>0$ for all $x,y \in \overline{\Omega}$ and for all $0<s<t$ such that $t^2 + s^2 < \rho_0$. First, assume that $x_0,y_0 \in \Omega$. Then, it follows that 
\[
\partial_t Z(P_0) \leq 0, \quad \partial_s Z(P_0) \leq 0,  \quad \nabla_{(x,y)} Z(P_0) = 0, \quad  D^2_{(x,y)}Z(P_0) \geq 0.
\]
By direct calculation, 
\begin{align*}
    \nabla_x Z(x,y;t,s) &= \nabla_x v(x,t) + \nabla_x\omega(x,y;t,s),\\
    \nabla_y Z(x,y;t,s) &= -\nabla_y v(y,s) +\nabla_y\omega(x,y;t,s),
\end{align*}
and thus $\nabla_{(x,y)} Z(P_0) = 0$ implies that
\begin{equation}\label{eq:schroe-grad}
    \nabla v(x_0,t_0) = -\nabla_x\omega(P_0), \quad \nabla v(y_0,s_0) = \nabla_y\omega(P_0).
\end{equation}
Next, we calculate
\begin{align*}
   \partial_t Z (x,y;t,s) &= \partial_t v(x,t) + \frac{\beta'(t)}{\beta(t)} + \partial_t \omega (x,y;t,s) + \varepsilon(t-s), \\
   \partial_s Z (x,y;t,s) &= -\partial_s v(y,s) - \frac{\beta'(s)}{\beta(s)} + \partial_s \omega (x,y;t,s) - \varepsilon(t-s). 
\end{align*}
Using that $\partial_t v = \Delta v + |\nabla v|^2-V$ in $\Omega$, $\partial_t Z(P_0) \leq 0$, and $\partial_s Z(P_0) \leq 0$ we have that 
\begin{align*}
    \Delta v(x_0,t_0) + |\nabla v(x_0,t_0)|^2 -V(x_0)+ \frac{\beta'(t_0)}{\beta(t_0)} + \partial_t \omega (P_0) + \varepsilon(t_0-s_0) &\leq 0, \\
    - \Delta v(y_0,s_0) - |\nabla v(y_0,s_0)|^2 + V(y_0) - \frac{\beta'(s_0)}{\beta(s_0)} + \partial_s \omega (P_0) - \varepsilon(t_0-s_0) &\leq 0.
\end{align*}
Then, using \eqref{eq:schroe-grad}, we obtain that
\begin{align*}
\begin{split}
\Delta v(x_0,t_0) &\leq -|\nabla_x\omega(P_0)|^2 - \partial_t \omega (P_0)+ V(x_0)  - \frac{\beta'(t_0)}{\beta(t_0)} - \varepsilon(t_0-s_0), 
\end{split}\\
\begin{split}
- \Delta v(y_0,s_0) &\leq |\nabla_y\omega(P_0)|^2 - \partial_s \omega (P_0) - V(y_0) + \frac{\beta'(s_0)}{\beta(s_0)} + \varepsilon(t_0-s_0).   
\end{split}
\end{align*}
Due to the assumed inequalities \eqref{eq:schroe-omega-ineq-1} and \eqref{eq:schroe-omega-ineq-2}, it follows that
\begin{align}
\Delta v(x_0,t_0) &\leq -\frac{\beta'(t_0)}{\beta(t_0)} - \varepsilon(t_0-s_0),   \label{eq:schroe-lap1}\\
- \Delta v(y_0,s_0) &\leq \frac{\beta'(s_0)}{\beta(s_0)} + \varepsilon(t_0-s_0).  \label{eq:schroe-lap2} 
\end{align}

Next, we make use of the inequality $D^2_{(x,y)}Z(P_0) \geq 0$. Let $x(\tau)$ and $y(\tau)$ be two paths in $\Omega$ such that
\begin{align*}
    x'(\tau) &= A(t_0) e_i, \quad x(\tau_0) = x_0, \\
    y'(\tau) &= A(s_0) e_i, \quad y(\tau_0) = y_0
\end{align*}
for some $\tau_0$, where $e_i$ is the $i^{\text{th}}$ standard basis vector. We have that
\begin{align*}
    \begin{split}
        \frac{d}{d\tau}Z(x(\tau),y(\tau);t,s) &= A(t_0) D_i v(x,t) - A(s_0) D_i v(y,s) \\&\hspace{1cm}+  A(t_0)\omega_{x_i}(x,y;t,s) + A(s_0)\omega_{y_i}(x,y;t,s)
    \end{split}
    \\
   \begin{split}
   \frac{d^2}{d\tau^2}Z(x(\tau),y(\tau);t,s) &=  A(t_0)^2D_{ii}v(x,t) - A(s_0)^2 D_{ii}v(y,s) +  A(t_0)^2 \omega_{x_i x_i}(x,y;t,s) \\ & \hspace{1cm}+ A(s_0)^2 \omega_{y_i y_i}(x,y;t,s) + 2A(t_0)A(s_0) \omega_{x_i y_i}(x,y;t,s)
   \end{split}
\end{align*}
Therefore, since $\frac{d^2}{d\tau^2}Z(P_0) \geq 0$, we have that
\[
\begin{split}
   &A(t_0)^2D_{ii}v(x_0,t_0) - A(s_0)^2 D_{ii}v(y_0,s_0) +  A(t_0)^2 \omega_{x_i x_i}(P_0) \\ & \hspace{3cm}+ A(s_0)^2 \omega_{y_i y_i}(P_0) + 2A(t_0)A(s_0) \omega_{x_i y_i}(P_0) \geq 0,
   \end{split}
\]
so that summing over $i$ yields 
\[
   \begin{split}
       A(t_0)^2&\Delta v(x_0,t_0) - A(s_0)^2 \Delta v(y_0,s_0) \\ &+ A(t_0)^2\Delta_x \omega(P_0) + A(s_0)^2\Delta_y \omega(P_0) + 2A(t_0)A(s_0) \sum_{i=1}^{d} \omega_{x_i y_i}(P_0) \geq 0.
   \end{split}
\]
By using the inequality \eqref{eq:schroe-omega-ineq-3}, as well as \eqref{eq:schroe-lap1} and \eqref{eq:schroe-lap2}, we obtain  
\[
0 > -\varepsilon(A(t_0)^2 - A(s_0)^2)(t_0-s_0) \geq 0,
\]
which is a contradiction. Therefore, such a point $P_0$ with $x_0, y_0 \in \Omega$ cannot exist. Thus, either one or both of $x_0$, $y_0$ must be on the boundary $\partial \Omega$. 

Without loss of generality, assume $x_0 \in \partial \Omega$. Then 
\begin{equation}\label{eq:Z-normal-heat}
    \left.\frac{d}{d\tau} Z(x_0 - \tau \nu(x_0),y_0;t_0,s_0)\right|_{\tau=0} = -\nabla v(x_0,t_0) \cdot \nu(x_0) - \nabla_x\omega(P_0)\cdot \nu(x_0) \leq 0,
\end{equation}
where we have used that $\nabla v(x_0,t_0) \cdot \nu(x_0)=0$ by the Neumann condition on $v$ as well as the assumption \eqref{eq:V-convexity-compact} on $\omega$. One of two cases must occur: either the inequality in \eqref{eq:Z-normal-heat} is strict, or we have equality.

Firstly, if \[\left.\frac{d}{d\tau} Z(x_0 - \tau \nu(x_0),y_0;t_0,s_0)\right|_{\tau=0}<0,\] since $Z$ is at least continuously differentiable in the first argument, there is some $\tau^* > 0$ such that $\frac{d}{d\tau} Z(x_0 - \tau \nu(x_0),y_0;t_0,s_0)<0$ for all $\tau \in [0,\tau^*)$. Integrating over this interval, we have \[Z(x_0-\tau^*\nu(x_0),y_0;t_0,s_0) < Z(x_0,y_0;t_0,s_0),\]
which contradicts $x_0$ minimising $Z(x,y_0;t_0,s_0)$.

The other possibility is that \begin{equation}\label{eq:zero}\left.\frac{d}{d\tau} Z(x_0 - \tau \nu(x_0),y_0;t_0,s_0)\right|_{\tau=0}=0.\end{equation} Considering $Z(x)=Z(x,y_0;t_0;s_0)$ as a function only of $x$, then \eqref{eq:zero} is equivalent to $\nabla_x Z(x_0)\cdot \nu(x_0) = 0$. Since $Z(x)$ has a minimum at $x=x_0$, then all derivatives of $Z(x)$ at $x_0$ in directions tangential to $\partial \Omega$ at $x_0$ must also be 0. Therefore, we conclude that $\nabla_xZ(x_0)=\nabla_x Z(P_0) = 0$.

Now consider whether $y \in \partial \Omega$ or $y \in \Omega$. If $y_0 \in \partial \Omega$, we could repeat our previous arguments to obtain $\nabla_yZ(P_0)=0$. If $y_0 \in \Omega$, then the function $Z(y)=Z(x_0,y;t_0,s_0)$ attains a minimum at $y=y_0$ and therefore $\nabla_yZ(P_0)=0$. In either case, we have that $\nabla_x Z(P_0) = 0$ and $\nabla_yZ(P_0)=0$, so $\nabla_{(x,y)}Z(P_0) =0$. Since $Z$ has a minimum at $P_0$, it follows that $D^2_{(x,y)}Z(P_0) \geq 0$. Combining these facts with $\partial_tZ(P_0) \leq 0$ and $\partial_sZ(P_0) \leq 0$, we may repeat the argument from when $x_0,y_0 \in \Omega$ to arrive at a contradiction. 

Therefore, $Z>0$ for all $x,y \in \overline{\Omega}$ and all $0<s<t$. Finally, taking $\varepsilon \rightarrow 0$ shows that the Harnack inequality \eqref{eq:heat-harnack-omega} holds.
\end{proof}

We now explain our result for the full-space problem.

\begin{proof}[Proof of Theorem \ref{thm:Euclidean-Schroedinger-full}]
Suppose that  $(\Omega_n)_{n \geq 1}$ is an increasing sequence of smooth bounded domains $\Omega_n \subseteq \R^d$ with $\bigcup_{n \geq 1} \Omega_n = \R^d$ such that \eqref{eq:V-convexity} holds. Let $u$ be a positive solution of \eqref{eq:schroedinger} with initial value $u_{0}:=u(x,0) \in L^2(\R^d)$ and define $u_{0n}:= u_{0}|_{\Omega_{n}}$ for all $n \geq 1$. Let $A_n$ be the operator $-\Delta+V$ with homogeneous Neumann boundary conditions and let $u_n$ be the positive solution to the problem 
\[
   \begin{cases}
   \begin{alignedat}{2}
       \partial_t u_{n} +A_n u&= 0&& \qquad \text{in $\Omega \times (0,\infty)$,} \\
         u_{n} (\cdot,0) &= u_{0n} && \qquad \text{in $\Omega _{n}$.}
   \end{alignedat}
\end{cases} 
\]
By Theorem \ref{thm:Euclidean-schrödinger}, $u_{n}$ satisfies the Harnack inequality 
\begin{equation}\label{eq:harnack-n}
u_{n}(x,t) \geq u_{n}(y,s) \frac{\beta(s)}{\beta(t)} e^{-\omega(x,y;t,s)}
\end{equation}
for all $x,y \in \overline{\Omega_n}$ and $0<s<t$.
We note that $A_n$ can be realised in $L^2(\Omega_n)$ by a continuous, elliptic form, and is therefore a quasi-$m$-accretive operator on $L^2(\Omega_n)$. The extension $\tilde{A}_n$ of $A_n$ by zero remains a quasi-$m$-accretive operator on $L^2(\R^d)$. Similarly, the operator $A=-\Delta + V$ is quasi-$m$-accretive on $L^2(\R^d)$. Moreover, it is not difficult to see that with $\tilde{A_n}$ converges to $A$ in $L^2(\R^d)$ in the resolvent sense. Therefore, using standard arguments from linear semigroup theory (see, for example, Kato \cite[Theorem 2.16 of Chapter 9]{kato}), it follows that $\tilde{u}_n(t) \to u(t)$ in $L^2(\R^d)$ uniformly for $t \in [0,T]$ and for every $T\geq 0$. In addition, using interior Schauder estimates (see, for example, Friedman \cite{friedman-parabolic}), one obtains that $u_n(x,t) \to u(x,t)$ locally uniformly in $\R^d \times (0,\infty)$. Therefore, taking $n \to \infty$ in the Harnack inequality \eqref{eq:harnack-n} yields the claimed result.    
\end{proof}

\subsection{An optimal $\omega$ and comparison results}\label{sec:optimal-omega}

The function $\omega$ in Theorem \ref{thm:Euclidean-Schroedinger-full} and Theorem \ref{thm:Euclidean-schrödinger} could potentially be any function satisfying the conditions of these theorems, especially the inequalities \eqref{eq:schroe-omega-ineq-1}--\eqref{eq:V-convexity}. However, an optimal candidate is the function $\omega$ defined by
\begin{equation}\label{eq:optimal-omega-def}
\omega(x,y;t,s):= \min_{\gamma \in \Gamma_{x,y}} E[\gamma;t,s]
\end{equation}
for the energy 
\begin{equation}\tag{\ref{eq:definition-of-E}}
E[\gamma;t,s]:= \frac{1}{4(t-s)}\int_0 ^1 |\dot{\gamma}|^2 \dd \tau + (t-s) \int_0^1 V(\gamma(\tau)) \dd \tau.   
\end{equation} This function $\omega$ is already well-known from the work of Li-Yau \cite{LiYau1986}. It occurs naturally while performing space-time integration to derive the Harnack inequality \eqref{eq:li-yau-harnack} from its differentiated form \eqref{eq:li-yau-diff-harnack}. The function $\omega$ defined in \eqref{eq:optimal-omega-def} is also closely related to the fundamental solution $\Gamma$ of the equation \eqref{eq:schroedinger}. In particular, it was proven first in the Euclidean case $M=\R^d$ by Simon \cite{Simon-1983} and later on a complete manifold $M$ by Li and Yau \cite{LiYau1986}, that $\omega$ can be expressed via the limit \[
    \omega(x,y;t,s) = -\lim_{\lambda \to 0} \frac{\log \Gamma_\lambda(x-y,\frac{t-s}{\lambda})}{\lambda},
    \]
    where $\Gamma_\lambda$ denotes the fundamental solution of the equation
    \begin{equation*}
        \partial_t u = \Delta u - \lambda^2V u \qquad \text{in $M\times (0,\infty)$}.
    \end{equation*}
Importantly, it was also demonstrated in \cite{LiYau1986} that the function $\omega$ defined in \eqref{eq:optimal-omega-def} satisfies both \eqref{eq:schroe-omega-ineq-1} and \eqref{eq:schroe-omega-ineq-2} with equality. In addition, as explained in the introduction, the inequalities \eqref{eq:V-convexity} are satisfied on any $V$-convex domain. However, in general, it is not clear that an inequality of the form 
\begin{equation}
    A(t)^2\Delta_x\omega +A(s)^2\Delta_y \omega + 2A(t)A(s) \sum_{i=1}^{d} \omega_{x_i y_i} \leq A(t)^2 \frac{\beta'(t)}{\beta(t)}-A(s)^2 \frac{\beta'(s)}{\beta(s)}\tag{\ref{eq:schroe-omega-ineq-3}}
\end{equation}
should be satisfied by this particular function $\omega$. We now give some important examples, where $\omega$ can be computed explicitly and appropriate functions $A$ and $\beta$ can be identified, which make inequality \eqref{eq:schroe-omega-ineq-3} true.

\begin{example}[Heat equation]
In the case of the heat equation ($V \equiv 0$), the optimal choice of $\omega$ is the function \[
    \omega(x,y;t,s):= \frac{|x-y|^2}{4(t-s)},
    \]
    which is the value attained by $E$ along the straight line segment, i.e. geodesic, connecting $x$ and $y$. In addition, $\omega$ satisfies \eqref{eq:schroe-omega-ineq-3} with equality if $A(\tau) = \tau$ and $\beta(\tau) = \tau^{d/2}$.
    With this choice, we recover the classical parabolic Harnack inequality 
    \begin{equation}\label{eq:classic-Harnack}
    u(x,t)\geq u(y,s) \left(\frac{s}{t} \right)^{d/2} e^{-\frac{|x-y|^2}{4(t-s)}}.
\end{equation} 
    
    We remark that \eqref{eq:classic-Harnack} is considered sharp, in the sense that it is satisfied by the fundamental solution 
    \[
    \Gamma(x,t) = (4 \pi t)^{-d/2}e^{\frac{-|x|^2}{4t}}
    \]
    with equality on the set of points \[
    \{(x,y;t,s) \in \R^d \times \R^d \times S \mid sx=ty \}.
    \] 
\end{example}

\begin{example}[Quadratic potential]\label{ex:quadratic}
Another noteworthy case to consider is that of the quadratic potential \[V(x) = C_1^2|x-a|^2+C_2\] for $a \in \R^d$ and constants $C_1,C_2 \ne 0$.  
In this case, one can compute explicitly that the energy $E[\gamma;t,s]$ given by \eqref{eq:definition-of-E} is minimised by the curve defined by \[
\begin{split}
    \gamma_0(\tau)= a+\frac{1}{\sinh(2C_1(t-s))} \Big(\sinh&(2C_1\tau (t-s))(x-a) \\&+ \sinh(2C_1(1-\tau) (t-s))(y-a)\Big)
\end{split}
\]
for all $\tau \in [0,1]$. It then follows that 
\[
\begin{split}
\omega (x,y;t,s) &= \frac{C_1}{2}\left(\frac{|x-y|^2}{\sinh(2C_1(t-s))}+(|x-a|^2 + |y-a|^2)\tanh(C_1(t-s))\right) \\&\hspace{1cm}+ C_2(t-s).   
\end{split}
\]
One can verify that this choice of $\omega$ satisfies \eqref{eq:schroe-omega-ineq-3} with equality if we set \[A(t) = \sinh(2C_1t), \qquad \text{and} \qquad \beta(t) = \sinh^{d/2}(2C_1t).\] Then, Theorem \ref{thm:Euclidean-Schroedinger-full} applied with $\Omega_n:= B_n(a)$ yields that every positive solution $u$  
of 
\begin{equation}\label{eq:schroe-quad-potential}
\partial_t u = \Delta u-(C_1^2|x-a|^2+C_2) u \qquad \text{in $\R^d \times (0,\infty)$}
\end{equation} satisfies the Harnack inequality \begin{equation}\label{eq:schroe-harnack-quad-potential}
  u(x,t) \geq u(y,s)\left(\frac{\sinh(2C_1s)}{\sinh(2C_1t)}\right)^{d/2}e^{-\omega(x,y;t,s)}  
\end{equation}
for every $x,y \in \R^d$ and $0 < s <t$.

We note that inequality \eqref{eq:schroe-harnack-quad-potential} is sharp in a similar sense to how the classical result for the heat equation without potential is sharp. For $a=0$, $C_2=0$, the fundamental solution of \eqref{eq:schroe-quad-potential} is known from \cite{Raich-Tinker-2016} to have the explicit formula
\[
    \Gamma(x,t) = \left(\frac{2\pi}{C_1}  \sinh(2C_1t)\right)^{-d/2}e^{\frac{-C_1|x|^2}{2\tanh(2C_1t)}}.
\]
One can verify that this solution satisfies \eqref{eq:schroe-harnack-quad-potential} with equality on the set of points \[\{ (x,y;t,s) \in \R^d \times \R^d \times S \mid \sinh(2Cs)x = \sinh(2Ct)y\}.\] This is an improvement of the inequality \eqref{eq:li-yau-harnack} obtained by Li and Yau \cite{LiYau1986}.
\end{example}

While it is difficult to determine whether for a given potential, the function $\omega$ defined in \eqref{eq:optimal-omega-def} satisfies \eqref{eq:schroe-omega-ineq-3}, the following comparison theorem can aid in finding suitable functions $A$ and $\beta$, which make \eqref{eq:schroe-omega-ineq-3} hold. We mention that throughout the remainder of this section, it is more convenient to work with the reparametrised energy 
\begin{equation}\label{eq:repara-energy}
E[\gamma;t,s]:= \frac{1}{4} \int_s^t |\dot{\gamma}|^2 \dd \tau + \int_s^t V(\gamma(\tau)) \dd \tau
\end{equation}
defined for curves $\gamma:[s,t] \to \R^d$ for $0<s<t$. Importantly, the minimum value denoted by $\omega$ does not depend on which definition we use for $E$.

\begin{theorem}[Comparison theorem]\label{thm:Comparison}
Let $\Omega \subseteq \R^d$ be open and suppose that $V_1,V_2 \in C^2(\Omega)$ are two potential functions bounded from below, which satisfy $\Delta V_1 \leq \Delta V_2$ in $\Omega$. Assume that there exist strictly increasing, differentiable functions $\beta:[0,\infty) \to [0,\infty)$ satisfying $\beta(0)=0$ and $A:[0,\infty) \to [0,\infty)$ such that \[
\int_s^t \frac{d}{2}A'(\tau)^2 + A(\tau)^2 \Delta V_2(\gamma_0(\tau)) \dd \tau \leq A(t)^2 \frac{\beta'(t)}{\beta(t)}-A(s)^2 \frac{\beta'(s)}{\beta(s)}
\]
for all $V_2$-geodesics $\gamma_0$ and $0<s<t$. Then the function $\omega_1$ defined by \eqref{eq:optimal-omega-def} for the potential $V=V_1$ satisfies \[
A(t)^2\Delta_x\omega_1 +A(s)^2\Delta_y \omega_1 + 2A(t)A(s) \sum_{i=1}^{d} {\omega_1}_{x_i y_i} \leq A(t)^2 \frac{\beta'(t)}{\beta(t)}-A(s)^2 \frac{\beta'(s)}{\beta(s)}
\]
in $\overline{\Omega} \times\overline {\Omega} \times  S$. Moreover, every positive solution $u$ to \eqref{eq:schroe-neumann-rd} with potential $V=V_1$ on a domain $\Omega_1 \subseteq \Omega$ satisfies the Harnack inequality 
\begin{equation}\label{eq:harnack-comparison}
   u(x,t) \geq u(y,s) \frac{\beta(s)}{\beta(t)} e^{-\omega_1(x,y;t,s)} 
\end{equation}
for all $x,y \in \overline{\Omega_1}$ and $0<s<t$, where $\Omega_1$ is assumed to be $V_1$-convex if $\Omega \ne \R^d$ and $\Omega_1 = \R^d$ if $\Omega = \R^d$ is $V_1$-approximable.
\end{theorem}

This theorem follows immediately from a short lemma.

\begin{lemma}\label{lem:omega-second-order-ineq}
Let $V \in C^2(\Omega)$ be bounded from below and $A:[0,\infty) \to \R$ differentiable. Then $\omega$ as defined by minimising the energy \eqref{eq:repara-energy} satisfies    
\begin{equation}\label{eq:lemma}
\begin{split}
   A(t)^2 \Delta_x \omega + A&(s)^2 \Delta_y \omega + 2A(t)A(s) \sum_{i=1}^d\frac{\partial^2 \omega}{\partial x_i \partial y_i} \\&\hspace{0.5cm}\leq \int_s^t \frac{d}{2}A'(\tau)^2 + A(\tau)^2 \Delta V(\gamma_0(\tau)) \dd \tau 
\end{split}
    \end{equation}
for all $V$-geodesics $\gamma_0$ in $\Omega$ and $0<s<t$.
\end{lemma}

\begin{proof}
Fix $x,y \in \Omega$ and denote by $\gamma_0$ a path which minimises $E$. Consider a family of curves defined by \[\gamma(\tau,r) = \gamma_0(\tau) + r A(\tau) e_i.\]
In particular, $\gamma(\tau,0) = \gamma_0(\tau)$ and the variation of the endpoints of the curves is described by 
\begin{align*}
x(r)&:=\gamma(t,r)=x + r A(t) e_i,\\
y(r)&:=\gamma(s,r)=y + r A(s) e_i.
\end{align*} 
Now, for any $r$, \[\omega(x(r),y(r);t,s) = \min_{\gamma \in \Gamma_{x(r),y(r)}} E[\gamma;t,s] \leq E[\gamma(\tau,r);t,s],\] where equality is attained for $r = 0$. Therefore, it follows that as a function of $r$, $\omega(x(r),y(r);t,s) - E[\gamma(\tau,r);t,s]$ is negative and attains a maximum value of $0$ at $r=0$. Therefore, \[\left.\frac{d^2}{dr^2}\omega(x(r),y(r);t,s) \right\rvert_{r=0} \leq \left.\frac{d^2}{dr^2}E[\gamma;t,s]\right\rvert_{r=0}.\] 
Computing both sides of this inequality, we obtain  
\[
\begin{split}
  A(t)^2 \frac{\partial^2 \omega}{\partial x_i^2} + A&(s)^2 \frac{\partial^2 \omega}{\partial y_i^2} + 2A(t)A(s) \frac{\partial^2 \omega}{\partial x_i \partial y_i} \\&\hspace{0.5cm}\leq \int_s^t \frac{1}{2}A'(\tau)^2 + A(\tau)^2 \frac{\partial^2}{\partial x_i^2}V(\gamma_0(\tau)) \dd \tau.  
\end{split}
\]
The claimed inequality \eqref{eq:lemma} follows by summing over $i$.
\end{proof}

\begin{proof}[Proof of Theorem \ref{thm:Comparison}]
    Applying Lemma \ref{lem:omega-second-order-ineq} to $\omega_1$, we have
    \begin{align*}
        A(t)^2 \Delta_x \omega_1 + A&(s)^2 \Delta_y \omega_1 + 2A(t)A(s) \sum_{i=1}^d\frac{\partial^2 \omega_1}{\partial x_i \partial y_i} \\ &\leq \int_s^t \frac{d}{2}A'(\tau)^2 + A(\tau)^2 \Delta V_1(\gamma_0(\tau)) \dd \tau\\
        &\leq \int_s^t \frac{d}{2}A'(\tau)^2 + A(\tau)^2 \Delta V_2(\gamma_0(\tau)) \dd \tau\\
        &\leq A(t)^2 \frac{\beta'(t)}{\beta(t)}-A(s)^2 \frac{\beta'(s)}{\beta(s)}
    \end{align*}
Hence the function $\omega_1$ corresponding to $V_1$ satisfies \eqref{eq:schroe-omega-ineq-3} with the functions $A$ and $\beta$ borrowed from $V_2$. All other necessary assumptions from Theorem~\ref{thm:Euclidean-Schroedinger-full} (if $\Omega = \R^d$) and Theorem~\ref{thm:Euclidean-schrödinger} (if $\Omega \ne \R^d$) hold and the Harnack inequality \eqref{eq:harnack-comparison} follows as a consequence.
\end{proof}

Applying Theorem \ref{thm:Comparison} with the quadratic potential $V_2 = C^2 |x|^2$, we find that a Harnack inequality holds for the class of potentials $V$ with $\Delta V$ bounded from above.

\begin{corollary}\label{cor:comparison-with-quad-pot}
    Let $\Omega \subseteq \R^d$ and $V \in C^2(\Omega)$ be bounded from below, such that either $\Omega$ is $V$-convex $(\Omega \ne \R^d)$ or $\Omega=\R^d$ is $V$-approximable. Suppose there is a constant $C\ne 0$ such that $\Delta V \leq 2dC^2$. 
Then, every positive solution $u$ of \eqref{eq:schroe-neumann-rd} satisfies
\begin{equation}
u(x,t) \geq u(y,s) \left(\frac{\sinh(2Cs)}{\sinh(2Ct)}\right)^{d/2} e^{-\omega(x,y;t,s)}
\end{equation}
for every $x,y \in \overline{\Omega}$ and $0<s<t$, where $\omega$ is obtained by minimising the energy \eqref{eq:repara-energy}. \end{corollary}


\section{The manifold case}\label{sec:schroedinger}
In Theorem \ref{thm:Euclidean-Schroedinger-full} and Theorem \ref{thm:Euclidean-schrödinger}, we assume that the function $\omega$ admits enough regularity to satisfy the inequalities \eqref{eq:schroe-omega-ineq-1}---\eqref{eq:V-convexity} pointwise. This assumption is not unreasonable to make, since as outlined in Section \ref{sec:optimal-omega}, it is possible to find such a regular function $\omega$ for many important potentials $V$. However, for general Riemannian manifolds $M$, these regularity properties of $\omega$ might no longer be true. In fact, the function $\omega$ given by \eqref{eq:optimal-omega-def} might be merely Lipschitz in $x$ and $y$ (see \cite{LiYau1986}). To bypass this issue in our proofs, we work directly with the energy $E$ and analyse the function $Z$ under smooth variations of curves connecting points $x$ and $y$ in $M$. 

\medskip

As in Section \ref{sec:euclidean}, we start by considering the case, where $M$ is compact.

\begin{proof}[Proof of Theorem \ref{thm:Manifold-Schroedinger}]
    Let $u$ be a positive solution of \eqref{eq:schroe-neumann-manifold} and set $v=\log u$. Then $v$ satisfies 
\begin{equation*}\label{eq:schroedinger-v}
   \begin{cases}
        \partial_t v = \Delta v + |\nabla v|^2 -V\quad &\text{in $M \times (0,\infty)$,} \\
        \nabla v \cdot \nu = 0 \quad &\text{on $\partial M \times (0,\infty)$.}
    \end{cases} 
\end{equation*}
Fix $\varepsilon > 0$ and define $Z$  by
\[
Z(x,y;t,s):= v(x,t) - v(y,s) + \log \left( \frac{\beta(t)}{\beta(s)} \right) + \omega(x,y;t,s) + \frac{\varepsilon}{2}(t-s)^2 + \varepsilon.
\]
As in the proof of Theorem \ref{thm:Euclidean-schrödinger}, we note that $Z>0$ as $s \to 0^+$. Next, considering the behaviour of $Z$ as $t\to s^+$, if $x \ne y$, we can demonstrate that $\omega \to \infty$ as $t \to s^+$, which guarantees that $Z>0$ for $t$ near $s$. Indeed, if $\gamma_0$ minimises $E[\gamma;t,s]$, then assuming $V \geq -\alpha$ in $M$ for some $\alpha>0$, it follows that 
\begin{align*}
\omega(x,y;t,s) = E[\gamma_0,t,s] &= \frac{1}{4(t-s)}\int_0 ^1 |\dot{\gamma}_0|^2 \dd \tau + (t-s) \int_0^1 V(\gamma_0(\tau)) \dd \tau \\&\geq \frac{1}{4(t-s)}\int_0 ^1 |\dot{\gamma}_0|^2 \dd \tau - \alpha(t-s).
\end{align*}
Now, the Cauchy-Schwarz inequality implies that \[
L(\gamma_0)^2 = \left(\int_0 ^1 |\dot{\gamma}_0| \dd \tau\right)^2 \leq \int_0 ^1 |\dot{\gamma}_0|^2 \dd \tau.
\]
Therefore, \begin{equation} \label{eq:omega lower bound}
    \omega(x,y;t,s) \geq \frac{L(\gamma_0)^2}{4(t-s)} - \alpha(t-s)\geq \frac{d^2(x,y)}{4(t-s)}- \alpha(t-s).
\end{equation}
Since we currently assume $d(x,y)>0$, this inequality implies $\omega \to \infty$ as $t \to s^+$ as desired. Otherwise, if we assume $x=y$, we still obtain 
\[
\lim_{t \to s^+} \omega(x,x;t,s) \geq 0.
\]
Then as in the proof of Theorem \ref{thm:Euclidean-schrödinger}, we conclude that $Z>0$ for $t$ near $s$. 

We again proceed as in our earlier proof, by supposing $P_0=(x_0,y_0;t_0,s_0)$ is such that $Z(P_0)=0$ for $\rho_0 = t_0^2 + s_0^2$ taken minimally. For now, we assume that $x_0, y_0 \not\in \partial M$.

Next, we define a new function $\tilde{Z}: \C^{\infty}([0,1], M) \times S \to \R$ by setting
\[
\tilde{Z}(\gamma;t,s)= v(\gamma(1),t)-v(\gamma(0),s) + \log \left( \frac{\beta(t)}{\beta(s)} \right) + E[\gamma;t,s] + \frac{\varepsilon}{2}(t-s)^2 + \varepsilon,
\]
which we analyse under smooth variations of $\gamma$. Clearly, $Z$ and $\tilde{Z}$ are related by the inequality 
\[
\tilde{Z}(\gamma;t,s) \geq Z(\gamma(1),\gamma(0);t,s),
\]
and equality holds if $\gamma$ minimises $E[\gamma;t,s]$ over the set of all paths connecting $\gamma(0)$ and $\gamma(1)$ fixed. Thus, we can conclude that if $Z$ has a minimum at $(x_0,y_0;t_0,s_0)$, then $\tilde{Z}$ has a minimum at $(\gamma_0,t_0,s_0)$, where $\gamma_0$ is such that $E[\gamma_0;t_0,s_0] \leq E[\gamma;t_0,s_0]$ for all $\gamma \in \Gamma_{x,y}$. From this, we can conclude that 
\begin{align}
    \partial_t \tilde{Z}(\gamma_0;t_0,s_0) & \leq 0 \label{eq:manifold-t}\\
    \partial_s \tilde{Z}(\gamma_0;t_0,s_0) & \leq 0 \label{eq:manifold-s}\\
    \left.\frac{\partial}{\partial r}\tilde{Z}(\gamma(\tau,r);t_0,s_0)\right|_{r=0} & = 0\label{eq:manifold-grad}\\
    \left.\frac{\partial^2}{\partial r^2}\tilde{Z}(\gamma(\tau,r);t_0,s_0)\right|_{r=0} & \geq 0\label{eq:manifold-hess}
\end{align}
for all smooth variations $\gamma=\gamma(\tau,r)$ of $\gamma_0$, with $r\in (-\delta,\delta)$ for some $\delta > 0$ and $\gamma(\tau,0)=\gamma_0(\tau)$. While interpreting these conditions, we encounter the derivatives $\frac{d}{dr} E$ and $\frac{d^{2}}{dr^{2}} E$
of $E[\gamma;t,s]$ under the variation of $\gamma_0$. For later use, we note that 
\begin{align}
\begin{split}
    \frac{d}{dr} E[\gamma(\tau,r);t,s] &= \frac{1}{2(t-s)}\bigg( g(\gamma_{r},\dot{\gamma})|_{0}^1\\&\hspace{1cm}-\int _{0}^1 g(\nabla_{\tau}\dot{\gamma}-2(t-s)^{2}DV(\gamma),\gamma_{r}) \dd \tau  \bigg)
\end{split}
    \label{eq:first-variation}\\\nonumber\\
\begin{split}
    \frac{d^{2}}{dr^{2}} E[\gamma(\tau,r);t,s] &= \frac{1}{2(t-s)} \int_{0}^1 R(\gamma_{r},\dot{\gamma},\gamma_{r},\dot{\gamma}) \dd \tau \\& \hspace{1cm}+ \frac{1}{2(t-s)} \int_{0}^1g(\nabla_{\tau}\nabla_{r}\gamma_{r},\dot{\gamma})+|\nabla_{r}\dot{\gamma}|^{2} \dd \tau \\&\hspace{1cm}+(t-s)\int _{0}^1 g(D^{2}V(\gamma)\gamma_{r},\gamma_{r}) + g(DV(\gamma),\gamma_{rr}) \dd \tau
\end{split}\label{eq:second-variation}
\end{align}

Using \eqref{eq:manifold-grad} and \eqref{eq:first-variation}, we obtain that 
\[
\begin{split}
    g(\nabla v(x_0,t_0),&\gamma_r(1))-g(\nabla v(y_0,s_0),\gamma_r(0))\\&= \frac{-1}{2(t_0-s_0)}(g(\gamma_r(1),\dot{\gamma}_0(1))-g(\gamma_r(0),\dot{\gamma}_0(0))).
\end{split}
\]
for all smooth variations $\gamma(\tau,r)$ of $\gamma_0$. By considering variations with $\gamma_r(1)=0$, and then with $\gamma_r(0)=0$, it follows that
\begin{equation}
\nabla v(x_0,t_0) = \frac{-1}{2(t_0-s_0)}\dot{\gamma}_0(1)\quad  \text{and} \quad \nabla v(y_0,s_0) = \frac{-1}{2(t_0-s_0)}\dot{\gamma}_0(0).\label{eq:v-grad}
\end{equation}

Next, \eqref{eq:manifold-t} implies
\begin{align*}
    0 &\geq \partial_tv(x_0,t_0) + \frac{\beta'(t)}{\beta(t)}+ \frac{\partial}{\partial t}E[\gamma_0;t_0,s_0] + \varepsilon(t_0-s_0)\\
    &= \Delta v(x_0,t_0) + |\nabla v(x_0,t_0)|^2 -V(x_0) + \frac{\beta'(t)}{\beta(t)} + \frac{\partial}{\partial t}E[\gamma_0;t_0,s_0] + \varepsilon(t_0-s_0)
\end{align*}
However, using \eqref{eq:v-grad} and computing $\frac{\partial}{\partial t}E[\gamma_0;t_0,s_0]$ as in \cite{LiYau1986}, we see that \[\frac{\partial}{\partial t}E[\gamma_0;t_0,s_0] = - \frac{|\dot{\gamma}_0(1)|^2}{4(t_0-s_0)^2}+ V(x)= -|\nabla v(x_0,t_0)|^2+ V(x).\] Therefore, 
\begin{equation} \label{eq:v-laplacian1}
    \Delta v(x_0,t_0) \leq - \frac{\beta'(t)}{\beta(t)}-\varepsilon(t_0-s_0).
\end{equation}
By a similar argument, \eqref{eq:manifold-s} implies 
\begin{equation} \label{eq:v-laplacian2}
    -\Delta v(y_0,s_0) \leq \frac{\beta'(s)}{\beta(s)}+\varepsilon(t_0-s_0).
\end{equation}

We now wish to make use of our last condition \eqref{eq:manifold-hess}. Choose an orthonormal basis $(e_i(0))_{i=1}^d$ of $T_{y_0}M$ and parallel transport it along $\gamma_0$ to obtain an orthonormal basis $e_i:= e_i (\tau)$ at each point on the curve. We then define a particular variation of $\gamma_0$ by setting
\[
\gamma(\tau,r) = \exp_{\gamma_0(\tau)}(A(\tau t_0 + (1-\tau)s_0)r e_i).
\]
We note that $\gamma_r(1,r)=A(t_0) e_i$, $\gamma_r(0,r)=A(s_0) e_i$, and $\gamma_{rr}(\tau,r)=0$ for all $r \in (-\delta,\delta)$. Then, by computing $\left.\frac{\partial^2}{\partial r^2}\tilde{Z}(\gamma(\tau,r);t_0,s_0)\right|_{r=0}$ explicitly, we find that 
\begin{equation}\label{eq:sub-in-hessian}
    0 \leq A(t_0)^2 \nabla_i \nabla_i v(x_0,t_0) - A(s_0)^2 \nabla_i \nabla_i v(y_0,s_0) + \left.\frac{\partial^2}{\partial r^2}E[\gamma(\tau,r);t_0,s_0]\right|_{r=0}.
\end{equation}

Applying the formula \eqref{eq:second-variation}, we get that 
\begin{align*}
&\hspace{-1cm}\left.\frac{\partial^2}{\partial r^2}E[\gamma(\tau,r);t_0,s_0]\right|_{r=0} \\&= \frac{1}{2(t_0-s_0)}\bigg( \int_0^1 |(t_0-s_0) A'(\tau t_0 + (1-\tau)s_0) e_i|^2 \\& \hspace{1cm}+ A(\tau t_0 + (1-\tau)s_0)^2 R(e_i,\dot{\gamma}_0,e_i,\dot{\gamma}_0) \dd \tau\bigg) \\& \hspace{1cm}+ (t_0 -s_0) \int_0^1 A(\tau t_0 + (1-\tau)s_0)^2 \nabla_i \nabla_i V(\gamma_0) \dd \tau\\
&= \frac{1}{2}(t_0-s_0)\int_0^1 A'(\tau t_0 + (1-\tau)s_0)^2 \dd \tau  \\&\hspace{1cm}- \frac{1}{2(t_0-s_0)} \int_0^1 A(\tau t_0 + (1-\tau)s_0)^2 R(e_i,\dot{\gamma}_0,\dot{\gamma}_0,e_i) \dd \tau \\& \hspace{1cm}+ (t_0 -s_0) \int_0^1 A(\tau t_0 + (1-\tau)s_0)^2 \nabla_i \nabla_i V(\gamma_0) \dd \tau
\end{align*}
Inserting this back into \eqref{eq:sub-in-hessian} and summing over $1 \leq i \leq d$, we have that 
\[
\begin{split}
  0 \leq A(t_0)^2 \Delta v(x_0,t_0) - A(s_0)^2 \Delta v(y_0,s_0) + \frac{d}{2}(t_0-s_0)\int_0^1 A'(\tau t_0 + (1-\tau)s_0)^2 \dd \tau \\ - \frac{1}{2(t_0-s_0)} \int_0 ^1 A(\tau t_0 + (1-\tau)s_0)^2 \Ric(\dot{\gamma}_0,\dot{\gamma}_0) \dd \tau \\ + (t_0 -s_0) \int_0^1 A(\tau t_0 + (1-\tau)s_0)^2 \Delta V(\gamma_0) \dd \tau.  
\end{split}
\]
But by assumption, $\Ric(\dot{\gamma}_0,\dot{\gamma}_0) \geq 0$, so this reduces to 
\begin{equation}\label{eq:comparison-manifold}
\begin{split}
  0 \leq A(t_0)^2 \Delta v(x_0,t_0) - A(s_0)^2 \Delta v(y_0,s_0) + \frac{d}{2}(t_0-s_0)\int_0^1 A'(\tau t_0 + (1-\tau)s_0)^2 \dd \tau \\ + (t_0 -s_0) \int_0^1 A(\tau t_0 + (1-\tau)s_0)^2 \Delta V(\gamma_0) \dd \tau.   
\end{split}
\end{equation}
Inserting the inequalities \eqref{eq:v-laplacian1}  \eqref{eq:v-laplacian2}, and \eqref{eq:second-order-assumption-manifold} into \eqref{eq:comparison-manifold} yields 
\[
0 \leq -\varepsilon(t_0-s_0)(A(t_0)^2 - A(s_0)^2) < 0.
\]
This is a contradiction since $A$ is assumed to be strictly increasing. If $\partial M = \emptyset$, then this concludes the proof. Otherwise, at least one of the points $x_0,y_0$ needs to lie on the boundary $\partial M$. 

Suppose $y_0 \in \partial M$. Let $e(0):=-\nu(y_0) \in T_{y_0}M$ and parallel transport $e(0)$ along $\gamma_0$, so that $e(\tau) \in T_{\gamma_0(\tau)}M$. We choose a variation $\gamma:[0,1] \times [0,\delta)\to M$ of $\gamma_0$ defined by 
\begin{equation}\label{eq:variation}
\gamma(\tau,r) = \exp_{\gamma_0(\tau)}(r(1-\tau)e(\tau)).
\end{equation}
Then $\gamma_r(\tau,r) = (1-\tau)e(\tau)$ for all $r \in [0,\delta)$. In particular, $\gamma_r(0,r)=-\nu(y_0)$ and $\gamma_r(1,r)=0$. With the help of \eqref{eq:first-variation}, we see that 
\begin{align*}
    \left.\frac{\partial}{\partial r} \tilde{Z}(\gamma(\tau,r);t_0,s_0)\right|_{r=0} &= g(\nabla v(x_0,t_0),\gamma_r(1,0)) - g(\nabla v(y_0,s_0),\gamma_r(0,0)) \\ & \qquad + \frac{1}{2(t_0-s_0)}(g(\gamma_r(1,0),\dot{\gamma}_0(1))-g(\gamma_r(0,0),\dot{\gamma}_0(0)))\\
    &= g(\nabla v(y_0,s_0),\nu(y_0)) + \frac{g(\nu(y_0),\dot{\gamma}_0(0)}{2(t_0-s_0)}\\
    &= \frac{g(\nu(y_0),\dot{\gamma}_0(0))}{2(t_0-s_0)},
\end{align*}
where the last equality follows as a consequence of the Neumann condition on $v$. The assumption that the boundary $\partial M$ is $V$-convex implies $g(\nu(y_0),\dot{\gamma}_0(0)) \leq 0$ and so \[\left.\frac{\partial}{\partial r} \tilde{Z}(\gamma(\tau,r));t_0,s_0)\right|_{r=0} \leq 0.\] 
If this inequality is strict, we may apply a similar argument as in the proof of Theorem \ref{thm:Euclidean-schrödinger} to reach a contradiction. Therefore, we must have equality, which can occur only if $g(\nu(y_0),\dot{\gamma}_0(0))=0$, that is, if $\dot{\gamma}_0(0) \in T_{y_0} \partial M$. Since the manifold was assumed to be $V$-convex, Remark \ref{rem:V-convex} implies this is only possible if $\dot{\gamma}_0(0)=0$. Therefore, for any variation of $\gamma_0$, not necessarily the one defined by \eqref{eq:variation}, we have 
\begin{align*}
    \left.\frac{\partial}{\partial r} \tilde{Z}(\gamma(\tau,r);t_0,s_0)\right|_{r=0} &= g(\nabla v(x_0,t_0),\gamma_r(1,0)) - g(\nabla v(y_0,s_0),\gamma_r(0,0)) \\ & \qquad + \frac{1}{2(t_0-s_0)}g(\gamma_r(1,0),\dot{\gamma}_0(1)).
\end{align*}
In particular, for variations $\gamma(\tau,r):[0,1] \times (-\delta,\delta) \to M$ with $\gamma_r(1,0)=0$ we have 
\[
    \left.\frac{\partial}{\partial r} \tilde{Z}(\gamma(\tau,r);t_0,s_0)\right|_{r=0} = - g(\nabla v(y_0,s_0),\gamma_r(0,0)).
\]
If we then consider choosing $\gamma_r(0,0) \in T_{y_0} \partial M$, then we have
\[
    0=\left.\frac{\partial}{\partial r} \tilde{Z}(\gamma(\tau,r);t_0,s_0)\right|_{r=0} = - g(\nabla v(y_0,s_0),\gamma_r(0,0)).
\]
Since the Neumann boundary condition implies $g(\nabla v(y_0,s_0),\gamma_r(0,0))=0$ for all $\gamma_r(0,0) \in (T_{y_0} \partial M)^{\perp}$, we conclude that \[g(\nabla v(y_0,s_0),\gamma_r(0,0))=0\]
for all choices of $\gamma_r(0,0) \in T_{y_0} M$ and therefore
\[\nabla v(y_0,s_0)=0 = \frac{-1}{2(t_0-s_0)}\dot{\gamma}_0(0)\]
holds as before. It follows that  
\[
\left.\frac{\partial}{\partial r} \tilde{Z}(\gamma(\tau,r);t_0,s_0)\right|_{r=0} = g(\nabla v(x_0,t_0),\gamma_r(1,0)) + \frac{1}{2(t_0-s_0)}g(\gamma_r(1,0),\dot{\gamma}_0(1))
\]
for any variation of $\gamma_0$. 

If $x_0 \not\in \partial M$, then we may consider variations $\gamma(\tau,r):[0,1] \times (-\delta,\delta) \to M$ with any choice of $\gamma_r(1,0) \in T_{x_0}M$. We then have 
\[
0=\left.\frac{\partial}{\partial r} \tilde{Z}(\gamma(\tau,r);t_0,s_0)\right|_{r=0} = g(\nabla v(x_0,t_0),\gamma_r(1,0))+ \frac{1}{2(t_0-s_0)}g(\gamma_r(1,0),\dot{\gamma}_0(1))
\]
and therefore \[
\nabla v(x_0,t_0) = \frac{-1}{2(t_0-s_0)}\dot{\gamma}_0(1)
\]
holds as before. If $x_0 \in \partial M$, by repeating the argument for $y_0$, it must be that $\dot{\gamma}_0(1)=0$ and $\nabla v(x_0,t_0)=0$ and again we recover the above identity. From this point on, we can continue to follow the argument as in the case that $x_0,y_0 \not \in \partial M$ to again reach a contradiction.
\end{proof}

\begin{proof}[Proof of Theorem \ref{thm:Manifold-Schroedinger-full}]
Since $M$ is assumed to be $V$-approximable, we may assume there exists a sequence $(M_n)_{n \geq 1}$ of $V$-convex domains approximating $M$. From here, the argument follows precisely as in the proof of Theorem \ref{thm:Euclidean-Schroedinger-full} with $\R^d$ replaced by $M$ and $\Omega_n$ replaced by $M_n$.
\end{proof}

\begin{remark}\label{rem:manifold-comparison}
In light of \eqref{eq:comparison-manifold}, a comparison result similar to Theorem \ref{thm:Comparison} applies in the manifold setting as well. Suppose that for two potential functions $V_1,V_2 \in C^2(M)$ bounded from below, one has that $\Delta V_1 \leq \Delta V_2$, and suitable functions $A$, $\beta$ making \eqref{eq:second-order-assumption-manifold} true for $V_2$ are known. Then choosing the same $A$ and $\beta$ makes \eqref{eq:second-order-assumption-manifold} true for $V_1$ as well. In particular, for any potential $V$ such that $\Delta V \leq 2dC^2$, we may once again make a comparison with the quadratic potential $V=C^2|x|^2$ on $M=\R^d$ and choose $A$ and $\beta$ by \[
A(t) = \sinh(2Ct), \qquad \text{and} \qquad \beta(t) = \sinh^{d/2}(2Ct).\]
\end{remark}

\section{Heat equation with gradient drift and potential}\label{sec:grad-drift}
Let $M$ be a complete Riemannian manifold and $f\in C^4(M)$. We consider positive solutions of the equation \begin{equation}\tag{\ref{eq:heat-grad-drift-potential-full}}
\partial_t u = \Delta u -2\nabla f \cdot \nabla u - Vu \qquad \text{in $M\times (0,\infty)$}.
\end{equation}
Under the change of variables $v:=e^{-f}u$, \eqref{eq:heat-grad-drift-potential-full} is transformed into the Schrödinger equation
\begin{equation}\label{eq:drift-transformed}
    \partial_t v = \Delta v - \tilde{V}v \qquad \text{in $M \times (0,\infty)$}
\end{equation}
for the potential $\tilde{V}:=|\nabla f|^2 - \Delta f + V$. Therefore, if the conditions of Theorem \ref{thm:Manifold-Schroedinger-full} are satisfied for the potential $\tilde{V}$, we may conclude that every positive solution $v$ of \eqref{eq:drift-transformed} satisfies the Harnack inequality 
\begin{equation*}\label{eq:drift-transformed-harnack}
v(x,t) \geq v(y,s) \frac{\beta(s)}{\beta(t)} e^{-\tilde{\omega}(x,y;t,s)},
\end{equation*}
where \[
\tilde{\omega}(x,y;t,s):= \min_{\Gamma_{x,y}} \frac{1}{4(t-s)}\int_0 ^1 |\dot{\gamma}|^2 \dd \tau + (t-s) \int_0^1 (|\nabla f|^2 - \Delta f + V)(\gamma(\tau)) \dd \tau.
\]
Reversing the transformation, this leads to the Harnack inequality 
\begin{equation}\tag{\ref{eq:drift-transformed-harnack-og}}
u(x,t) \geq u(y,s) \frac{\beta(s)}{\beta(t)} e^{f(x)-f(y)-\tilde{\omega}(x,y;t,s)}
\end{equation}
satisfied by every positive solution $u$ of the original equation \eqref{eq:heat-grad-drift-potential-full}, which is Corollary \ref{cor:drift-full}. Corollary \ref{cor:drift} can be obtained similarly by applying Theorem \ref{thm:Manifold-Schroedinger}.

\begin{example}[Ornstein-Uhlenbeck with quadratic potential] For $M=\R^d$, we consider when $V=C_1^2|x|^2$ and $f=-\frac{C_2}{2}|x|^2$ for $C_1,C_2 \in \R$. These choices correspond to the equation
\[
\partial_t u = \Delta u + 2C_2x\cdot\nabla u - C_1^2|x|^2 u,
\]
which under the transformation $v=e^{-\frac{C_1}{4}|x|^2}$ becomes 
\[
\partial_t v = \Delta v - (C^2|x|^2 + dC_2)v
\]
for $C$ such that $C^2 = C_1^2 + C_2^2$. From Example \ref{ex:quadratic}, it follows that a positive solution $v$ of this equation satisfies 
\[
  v(x,t) \geq v(y,s)\left(\frac{\sinh(2Cs)}{\sinh(2Ct)}\right)^{d/2}e^{-\tilde{\omega}(x,y;t,s)} 
\]
for all $x,y \in \R^d$ and $0<s<t$ with 
\[
\tilde{\omega} (x,y;t,s) = \frac{C}{2}\left(\frac{|x-y|^2}{\sinh(2C(t-s))}+(|x|^2 + |y|^2)\tanh(C(t-s))\right) + dC_2(t-s).
\]
Reversing the transformation, we have that $u$ satisfies 
\[
  u(x,t) \geq u(y,s)\left(\frac{\sinh(2Cs)}{\sinh(2Ct)}\right)^{d/2}e^{\frac{C_2}{2} (|x|^2-|y|^2)-\tilde{\omega}(x,y;t,s)}. 
\]
One may compare this inequality with the results obtained in \cite{Negro}.
\end{example}

\section{Differential Harnack inequalities}\label{sec:diff-harnack}
Suppose $u$ is a positive function defined on a Riemannian manifold $M$ satisfying the Harnack inequality
\begin{equation*}
      u(x,t) \geq u(y,s) \frac{\beta(s)}{\beta(t)} e^{-\omega(x,y;t,s)}
\end{equation*} for a given function $\beta:[0,\infty) \to \beta:[0,\infty)$ and $\omega$ defined as in Theorem \ref{thm:Manifold-Schroedinger-full}. Then, after making the change of variables $v=\log u$, we have that \[
v(x,t) - v(y,s) + \log(\beta(t)) - \log(\beta(s)) + \frac{1}{4}\int_{s}^t |\dot\gamma|^2 \dd \tau + \int_s^t V(\gamma(\tau)) \dd \tau \geq 0
\]
for all $x,y \in M$, $0<s<t$, where $\gamma$ is a $V$-geodesic connecting $x$ and $y$ in $M$. For $h > 0$ small and $\Vec{e} \in T_y M$ to be determined later, set $x= \exp_y(h \Vec{e})$ and $t=s+h$. We write $\gamma_h:[s,s+h] \to M$ to represent the $V$-geodesic with $\gamma_h(s) = y$ and $\gamma_h(s+h) = x = \exp_y(h \Vec{e})$.
Then 
\[
\begin{split}
\frac{v(\gamma_h(s+h),s+h) - v(y,s)}{h} + \frac{\log(\beta(s+h)) - \log(\beta(s))}{h}& \\ + \frac{1}{4h}\int_{s}^{s+h} |\dot\gamma_h|^2 \dd \tau + \frac{1}{h}\int_s^{s+h} V(\gamma_h(\tau)) \dd \tau &\geq 0.
\end{split}
\]
Passing to the limit as $h \to 0^+$ yields
\[
\nabla v(y,s) \cdot \Vec{e} + \partial_s v(y,s) + \frac{\beta'(s)}{\beta(s)} + \frac{1}{4}|\Vec{e}|^2 + V(y) \geq 0.
\]
Writing $\frac{1}{4}|\Vec{e}|^2 +\nabla v(y,s) \cdot \Vec{e} = |\nabla v(y,s) + \frac{1}{2}\Vec{e}|^2 - |\nabla v(y,s)|^2$ and then choosing $\Vec{e} = -2\nabla v(y,s)$, we have that 
\[
\partial_s v(y,s) + \frac{\beta'(s)}{\beta(s)} - |\nabla v(y,s)|^2 + V(y) \geq 0,
\]
or equivalently if $u$ solves \eqref{eq:schroedinger}, \[
\Delta (\log u) \geq -\frac{\beta'(s)}{\beta(s)}
\]
in $M \times (0, \infty)$. Alternatively, if $u=e^{-f}\hat{u}$, where $\hat{u}$ solves \eqref{eq:heat-grad-drift-potential-full}, then 
\[
\Delta (\log \hat{u}) \geq \Delta f-\frac{\beta'(s)}{\beta(s)}
\]
in $M \times (0, \infty)$.

\bibliographystyle{plain}

\end{document}